\documentclass[preprint]{imsart}

\usepackage{a4wide}
\usepackage{amssymb}
\usepackage{amsfonts}
\usepackage{amsthm}
\usepackage{enumerate}
\usepackage{amsmath}
\usepackage{bbm}

\usepackage[french,english]{babel}

\newtheorem{theorem}{Theorem}

\newtheorem{lemma}[theorem]{Lemma}

\newtheorem{remark}{Remark}

\def\eps{\varepsilon}
\def\ZZ{\mathbb{Z}}
\def\NN{\mathbb{N}}
\def\N{\mathbb{N}}

\def\R{\mathbb{R}}

\def\PP{\mathbb{P}}
\def\P{\mathbb{P}}
\def\p{\mathbb{P}}
\def\EE{\mathbb{E}}
\def\F{\mathcal{F}}
\def\E{\mathcal{E}}

\def\L{\mathcal{L}}

\DeclareMathOperator{\var}{var}

\begin{document}

\begin{frontmatter}
\title{Longest common substring for random subshifts of finite type}
\runtitle{Longest common substring for random subshifts of finite type}
\thankstext{T1}{This work was partially supported by CNPq, by FCT project PTDC/MAT-PUR/28177/2017, with national funds, and by CMUP (UIDB/00144/2020), which is funded by FCT with national (MCTES) and European structural funds through the programs FEDER, under the partnership agreement PT2020.}

\begin{aug}
\author{\fnms{J\'er\^ome} \snm{Rousseau}
\ead[label=e3]{jerome.rousseau@ufba.br}%
\ead[label=u2,url]{www.sd.mat.ufba.br/\textasciitilde jerome.rousseau}}

\affiliation{Universidade do Porto and Universidade Federal da Bahia}
\address{Departamento de Matem\'atica, \\ Faculdade de Ci\^encias da Universidade do Porto,\\Rua do Campo Alegre, 687, \\4169-007 Porto, Portugal\\
\printead{e3}\\
\printead{u2}}
\address{Departamento de Matem\'atica, \\ Universidade Federal da Bahia,\\ Av. Ademar de Barros s/n, \\40170-110 Salvador, Brazil}

\runauthor{J. Rousseau}

\end{aug}
\begin{abstract}
In this paper, we study the behaviour of the longest common substring for random subshifts of finite type (for dynamicists) or of the longest common substring for random sequences in random environments (for probabilists).
We prove that, under some exponential mixing assumptions, this behaviour is linked to the R\'enyi entropy of the stationary measure. We emphasize that what we establish is a quenched result.
\end{abstract}

\selectlanguage{french}
\begin{abstract}
Dans cet article, nous \'etudions le comportement de la plus longue sous-cha\^ine commune pour des sous-shifts al\'eatoires de type fini (pour les dynamiciens) ou de la plus longue sous-cha\^ine commune pour des suites al\'eatoires en milieux al\'eatoires (pour les probabilistes). Nous prouvons que, sous des hypoth\'eses de m\'elange exponentiel, ce comportement est li\'e \`a l'entropie de R\'enyi de la mesure stationnaire. Nous soulignons que ce que nous \'etablissons est un r\'esultat fibr\'e.
\end{abstract}



\end{frontmatter}


\section*{Introduction}

To try and measure the similarity between sequences, one has to develop computational tools to compare the sequences (and to optimize the algorithm) and probabilistic tools to discern the significance of the relationship. Thus sequences comparison (and in particular sequences alignment and sequences matching) takes its roots in computer science and probability and has applications in areas as diverse as bioinformatics, geology, linguistics or social sciences. We refer the reader to \cite{RSW,W-book} for a broad introduction to sequences comparison (with a particular attention to biology).

One particularly relevant object in DNA comparison is the longest common substring, i.e. the longest string of DNA which appears in two (or more) strands.  
For example, for the following two strands
\[ACAATGAGAGGATGACCTTG\]
\[TGACTGTAACTGACACAAGC\]
a longest common substring is ACAA (TGAC is also a longest common substring) and is of length 4 when the total length of the strands is 20. A way to distinguish if this behaviour is common or rare is to obtain probabilistic results which allow us to understand the statistical significance of our comparison.

In this paper, we will concentrate on the behaviour of the length of the longest common substring when the length of the strings grows, more precisely, for two sequences $x$ and $y$, the behaviour, when $n$ goes to infinity, of
\[M_n(x,y)=\max\{m:x_{i+k}=y_{j+k}\textrm{ for $k=1,\dots,m$ and for some $0\leq i,j\leq n-m$}\}.\]
For sequences drawn randomly from the same alphabet, this problem was studied by Arratia and Waterman in \cite{AW}. More precisely, if each term of the sequences is drawn independently within some alphabet $\mathcal{A}$ with respect to some probability $\mathcal{P}$, then they proved that for $\mathcal{P}^\N\otimes\mathcal{P}^\N$-almost every $(x,y)\in\mathcal{A}^\NN\times\mathcal{A}^\NN$
\begin{equation*}
\lim_{n\rightarrow\infty}\frac{M_n(x,y)}{\log n}=\frac{2}{-\log p}
\end{equation*}
where $p=\sum_{a\in\mathcal{A}}\mathcal{P}(a)^2$.

They also proved the same result for independent irreducible and aperiodic Markov chains on a finite alphabet, and in this case $p$ is the largest eigenvalue of the matrix $[(p_{ij})^2]$ (where $[p_{ij}]$ is the transition matrix).

In fact, one can observe that in both case, $-\log p$ corresponds to the R\'enyi entropy of $\mu$ defined (provided that it exists) by
$${H}_2(\mu) = \underset{k \to \infty}{{\lim}}\frac{\log\sum \mu(C_k)^2}{-k}$$
where the sums are taken over all k-cylinders. Even if the existence of the R\'enyi entropy is not known in general, it was computed in some particular cases: Bernoulli shift, finite state Markov chains, Gibbs measure of a H\"older-continuous potential \cite{haydnvaienti} and infinite state Markov chains \cite{ciuperca}. The existence was also proved for $\phi$-mixing measures \cite{luc}, for weakly $\psi$-mixing processes \cite{haydnvaienti} and for $\psi_g$-regular processes \cite{AbCa, AbLa2}. 

Generalizations of the work \cite{AW} to sequences of different lengths, different distributions, more than two sequences, extreme value theory for sequence matching and distributional results can be found in e.g. \cite{AW1,AGW,AW2,AW3,KO,DKZ,Neu, Manson}. In a similar direction, one can also see \cite{ColletRedig,deb} (and references therein) where the authors investigate the growth rate of the maximal overlap in a string (i.e. the growth rate of the length of the longest repeated substring). We also refer to \cite{Shield,AV, AbLa2,KASW} for relatively close problems.

Recently, in \cite{BaLiRo}, the results of Arratia and Waterman were generalized to $\alpha$-mixing systems with exponential decay (and $\psi$-mixing with polynomial decay) and it was proved that if the R\'enyi entropy exists then for $\mu\otimes\mu$-almost every $(x,y)$
\[\lim_{n\rightarrow\infty}\frac{M_n(x,y)}{\log n}=\frac{2}{H_2(\mu)}.\]

Furthermore, it was also shown in this paper that a generalization of the longest common substring problem for dynamical systems is to study the behaviour of the shortest distance between two orbits, which is, for a dynamical system $(X,T,\mu)$, the behaviour, when $n$ goes to infinity, of
\begin{equation*}\label{defshortest}
m_n(x,y)=\min_{i,j=0,\dots,n-1}\left(d(T^ix,T^jy)\right).\end{equation*}
Moreover, a relation between $m_n$ and the correlation dimension of the invariant measure was proved.

It is natural to try and obtain the same type of results for random dynamical systems since they could model more precisely physical phenomena. For random sequences, this could correspond for example to a modification (e.g. a small perturbation) on the probability with which the letters of the alphabet are drawn (i.e. random sequences in random environments). For dynamical systems, this could correspond to adding some random noise or small perturbations while iterating the same transformation, or iterating different transformations drawn randomly within a family of transformations (see e.g. \cite{kifliu} for an introduction to random dynamical systems).

In \cite{colaro}, the behaviour of the longest common substring of encoded sequences (and of the shortest distance between observed orbits) were studied and a relation with the R\'enyi entropy of the pushforward measure was proved. In particular, it allows the authors to obtain annealed results on the shortest distance between orbits of random dynamical systems. 

Obtaining quenched results is much more delicate, in particular because generally the random maps do not have a common invariant measure. The first family of random dynamical systems to study and where one can hope to obtain results are random subshifts of finite type. Indeed, good mixing properties have been proved (see e.g. \cite{kif, boggun,KifLim, stad}) which allows to get other statistical properties (e.g. \cite{RSV,RT,HT} for the distribution of hitting times, \cite{FFV} for extreme value laws). Following this idea and the setting of these papers, we study here the behaviour of the longest common substrings for random subshifts of finite type (in probabilistic language, this corresponds to the longest common substring for random sequences in random environments) and prove a link with the R\'enyi entropy of the stationary measure.

The paper is organized as follows. In Section~\ref{sec:statement}, we will define random subshifts of finite type, explain our assumptions and give an upper bound (Theorem~\ref{thprinc}) and a lower bound (Theorem~\ref{thprinc2} and Theorem~\ref{thprinc3}) for the growth rate of the longest common substring for random subshifts. In Section~\ref{example}, we will apply our results to random Bernoulli shifts and random Gibbs measures. The proof of the theorems will be given in Section~\ref{secproof}.

\section{Statement of the main results}\label{sec:statement}

We first give the definition of a random subshift of finite type.
Let $(\Omega,\theta,\PP)$ be an invertible ergodic measure preserving system, set $X=\{1,\dots,N\}^\NN$  for some $N\in \NN$ and let $\sigma: X \to X$ denote the shift. Let  $b:\Omega \to \{1,\dots,N\}$ be a random variable. 
Let $A=\left\{A(\omega)=(a_{ij}(\omega)):\omega\in \Omega\right\}$ be a random transition matrix, i.e. for any $\omega\in\Omega$, $A(\omega)$ is a $b(\omega)\times b(\theta\omega)$-matrix with entries in $\{0,1\}$, at least one non-zero entry in each row and each column and such that $\omega\mapsto a_{ij}(\omega)$ is measurable for any $i\in\NN$ and $j\in\NN$.  
For any $\omega\in \Omega$
define the subset of the integers $X_\omega=\{1,\ldots,b(\omega)\}$ and
\begin{equation*}
\E_\omega =\{x=(x_0,x_1,\ldots)\colon x_i\in X_{\theta^i\omega} \text{ and } a_{x_i x_{i+1}}(\theta^i\omega)=1\text{ for all } i\in\NN\}\subset X, 
\end{equation*}
\begin{equation*}
\E = \{(\omega,x)\colon \omega\in\Omega,x\in\E_\omega\} \subset \Omega\times X.
\end{equation*}
We consider the random dynamical system coded by the skew-product $S : \E \to \E$ given by
$S(\omega,x)= (\theta \omega,\sigma x)$. Let $\nu$ be an $S$-invariant probability
measure with marginal $\PP$ on $\Omega$ and let $(\mu_\omega)_\omega$ denote
its decomposition on $\E_\omega$, that is, $d\nu(\omega,x)=d\mu_\omega(x)d\PP(\omega)$. 
The measures $\mu_\omega$ are called the \emph{sample measures}. Note $\mu_\omega(A)=0$ if $A\cap X_\omega=\emptyset$. We denote by $\mu=\int \mu_\omega \, d\PP$ the marginal of $\nu$ on $X$.

We emphasize that the sample measures are not invariant. However, since $\theta$ is invertible, by $\sigma$-invariance of $\nu$ and almost everywhere uniqueness of the 
decomposition $d\nu=d\mu_\omega \, d\mathbb{P} $, we get for $\PP$-almost every $\omega\in\Omega$,
\begin{equation}\label{eqinv}
(\sigma^i)_*\mu_\omega=\mu_{\theta^i\omega} \qquad\textrm{for all $i\in\NN$.}
 \end{equation}

For $y\in X$ we denote by $C_n(y)=\{z \in X : y_i=z_i \text{ for all } 
0\le i\le n-1\} $ the  \emph{$n$-cylinder} that contains $y$. Set $\F_0^n(X)$ as the sigma-algebra in $X$ 
generated by all the $n$-cylinders.

As explain in the introduction, for two sequences $x,y\in X$, we are interested in the asymptotic behaviour of the longest common substring, that is the behaviour of
\[M_n(x,y)=\max\{m:x_{i+k}=y_{j+k}\textrm{ for $k=1,\dots,m$ and for some $0\leq i,j\leq n-m$}\}.\]

We will show it is linked to the R\'enyi entropy of the stationary measure $\mu$. Thus, we define the lower and upper R\'enyi entropies of the measure $\mu$: 
$$\underline{H}_2(\mu) = \underset{k \to \infty}{\underline{\lim}}\frac{\log\sum \mu(C_k)^2}{-k} \ \  \mbox{and} \ \ \overline{H}_2(\mu) = \underset{k \to \infty}{\overline{\lim}}\frac{\log\sum \mu(C_k)^2}{-k} \ ,$$
where the sums are taken over all k-cylinders. When the limit exists we denote by ${H}_2(\mu)$ the common value.

To obtain our results, we will need information on the decay of the measure of cylinders, thus we define
\[h_0=\underset{k\rightarrow+\infty}{\underline{\lim}}\frac{\log\int_\Omega\underset{C_k}{\max}\mu_\omega(C_k) d\PP}{-k}\]
where the max is taken over all k-cylinders.

We will assume the following: there is a constant $a\in [0,1)$ and a function $\alpha(g)$ satisfying $\alpha(g)=\mathcal{O}(a^g)$ such that
for all $n,m$, $A\in\F_0^n(X)$ and $B\in\F_0^m(X)$:

\begin{itemize}
\item[(I)] the marginal measure $\mu$ satisfies
\[
\left|\mu(A\cap\sigma^{-g-n}A) -\mu(A)^2\right|\le \alpha(g);
\]
\item[(II)] (fibered exponential $\alpha$-mixing) for $\PP$-almost every $\omega\in\Omega$
\[
\left|\mu_\omega(A\cap\sigma^{-g-n}B) -\mu_\omega(A)\mu_{\theta^{n+g}\omega}(B)\right|\le \alpha(g).
\]

\end{itemize}

One can observe that assumption (I) is weaker than $\alpha$-mixing since in the intersection we only deal with the same cylinder $A$. We recall that the measure $\mu$ is $\alpha$-mixing if:
\begin{itemize}
\item[(I-a)] (exponential $\alpha$-mixing) the marginal measure $\mu$ satisfies
\[
\left|\mu(A\cap\sigma^{-g-n}B) -\mu(A)\mu(B)\right|\le \alpha(g)
\]
for all $m,n$, $A\in\F_0^n(X)$ and $B\in \F_0^{m}(X)$.
\end{itemize}

Before stating our results, we will consider the annealed case:

\begin{theorem}[Theorem 4.4 \cite{colaro}] \label{thprincann} If $0<\underline{H}_2(\mu)$, then
\[
\underset{n \to \infty}{\overline{\lim}}\frac{M_n(x,y)}{\log n} \leq \frac{2}{\underline{H}_2(\mu)} \textrm{ for $\nu
\otimes \nu$-almost every $((\omega,x),(\tilde\omega,y)) \in \E\times \E$.}
\]
 Moreover, if hypothesis (I-a) holds, then
\[
\underset{n \to \infty}{\underline{\lim}}\frac{M_n(x,y)}{\log n} \geq \frac{2}{\overline{H}_2(\mu)}  \textrm{ for $\nu
\otimes \nu$-almost every $((\omega,x),(\tilde\omega,y)) \in \E\times \E$.}\]
\end{theorem}

First of all, we observe that the statement of this theorem is slightly different that the one of Theorem 4.4 in \cite{colaro} since they consider more general dynamical systems and not only random subshifts of finite type. Nevertheless, one can adapt easily their results and proof to
obtain the theorem as stated here. 

One could wonder why the R\'enyi entropy of $\mu$ appears in these results (and not the R\'enyi entropy of $\nu$ for example). In fact, when studying $M_n$, we are not interested on the behaviour of the whole orbits $S^n(\omega,x)$ but only its projection on $X$ (called an observation of the dynamical system). More precisely, if $\pi:\E\rightarrow X$ denotes the canonical projection (i.e. $\pi(\omega,x)=x$), we study the behaviour of the image (or observation) of the orbits, that is $\pi(S^n(\omega,x))$. 
The idea of observing dynamical systems was developed in \cite{RS,Rousseau} to obtain annealed results for return times in random dynamical systems and for the shortest distance between random orbits in \cite{colaro}. Moreover, it was proved, that when observing dynamical systems, these quantities are linked with the dimension (or in our case the R\'enyi entropy) of the pushforward measure $\pi_*\nu$ (where $\pi_*\nu(.)=\nu(\pi^{-1}(.))$). Furthermore, in our random setting the pushforward measure $\pi_*\nu$ and the measure $\mu$ are equals (e.g. \cite[Proof of Theorem 8]{Rousseau}) and thus ${H}_2(\pi_*\nu)={H}_2(\mu)$.

Unfortunately, these technics only give annealed results, thus, in this paper, we will use different tools to obtain quenched results.

\begin{remark}We note that since $S$ is a dynamical system, one could apply (under the right assumptions) the results of \cite{BaLiRo} to study the shortest distance between orbits $m_n((\omega,x),(\tilde\omega,y))=\min_{i,j=0,\dots,n-1}\left(d(S^i(\omega,x),S^j(\tilde\omega,y))\right)$ and link it to the correlation dimension of $\nu$. Nevertheless, it will not give us precise informations on $M_n(x,y)$ since $m_n((\omega,x),(\tilde\omega,y))$ takes into account the distance between elements of the orbits of $(\omega,x)$ and $(\tilde\omega,y)$ while $M_n$ only considers elements of the orbits of $x$ and $y$.
\end{remark}

We present now the first main result of this section which gives an upper bound for the growth rate of the longest common substring. 

\begin{theorem} \label{thprinc} If $0<\underline{H}_2(\mu)\leq 2h_0$ and if hypothesis (I) and (II)  hold, then for $\PP$-almost every $\omega \in \Omega$,
\[
\underset{n \to \infty}{\overline{\lim}}\frac{M_n(x,y)}{\log n} \leq \frac{2}{\underline{H}_2(\mu)} \textrm{ for $\mu_\omega \otimes \mu_\omega$-almost every $(x,y) \in \E_\omega\times \E_\omega$.}
\]
\end{theorem}
One can notice that in the deterministic case \cite{BaLiRo} and in the annealed case, no mixing assumptions are needed to obtain the upper bound. As one can see in the proof of this theorem, the main problem and difference with the deterministic case is that the sample measures are not invariant which is the main reason to use mixing to obtain the upper bound (and the lower).

Moreover, one can observe that assuming $\underline{H}_2(\mu)\leq 2h_0$ is not a too restrictive assumption. Indeed, in the deterministic case this hypothesis is always satisfied (see e.g. \cite{haydnvaienti} in the proof of Theorem 1 (IV)). In the random setting, this assumption prohibits for example to have some sample measures with an extreme behaviour (relatively with the others).

To obtain a lower bound, we will need stronger assumptions: we will need $\alpha$-mixing for the measure $\mu$ and we will require some mixing properties for the base transformation $(\Omega, \theta, \P)$. 

First of all, we will treat the case when $(\Omega, \theta, \P)$ is a $\rho$-mixing two-sided shift, i.e. $\Omega=\mathcal{A}^\ZZ$ for some alphabet $\mathcal{A}$, $\theta$ is the shift and:

\begin{itemize}
\item[(III)] (exponential $\rho$-mixing) For all $n$ and for all $\psi\in L^2(\F_{-\infty}^n(\Omega))$ and $\phi\in L^2(\F_0^{\infty}(\Omega))$
\[
\left|\int_\Omega\psi.\phi\circ \theta^{n+g}\, d\P-\int_\Omega \psi d\P\int_\Omega\phi d \P\right|\leq \rho(g)\|\psi\|_2\|\phi\|_2
\]
with $\rho(n) =\mathcal{O}(a^{n})$.
\end{itemize}
Moreover, we will need that the sample measure $\mu_\omega$ of a cylinder of size $n$ does not depend on all the terms of $\omega$:
\begin{itemize}
\item[(IV)] there exists a function $h$ with $h(n)=\mathcal{O}(n)$ such that for $\P$-almost every $\omega$ and every cylinder $C\in\F_0^n(X)$, the function $\omega\mapsto\mu_\omega(C)$ belongs to $L^2(\F_{-h(n)}^{h(n)}(\Omega))$.
\end{itemize}
One can observe that it is quite simple to check if assumption (IV) is satisfied, however this assumption is restrictive and only enables us to work with some special family of sample measures. Nevertheless, if the system $(\Omega, \theta, \P)$ satisfies some stronger mixing assumption we will be able to work with more general families of sample measures. Thus, after the statement of the next theorem we will give an alternative couple of assumptions which also allows us to obtain a lower bound for the growth rate of the longest common substring.
\begin{theorem} \label{thprinc2}
 If $0<\underline{H}_2(\mu)\leq\overline{H}_2(\mu)< 2h_0$ and if hypothesis (I-a), (II), (III) and (IV) hold, then, for $\PP$-almost every $\omega \in \Omega$,
\[
\underset{n \to \infty}{\underline{\lim}}\frac{M_n(x,y)}{\log n} \geq \frac{2}{\overline{H}_2(\mu)}  \textrm{ for $\mu_\omega \otimes \mu_\omega$-almost every $(x,y) \in \E_\omega\times \E_\omega$.}
\]
Moreover, if the R\'enyi entropy exists, we get for $\PP$-almost every $\omega \in \Omega$,
\[
\lim_{n \to \infty}\frac{M_n(x,y)}{\log n} = \frac{2}{{H}_2(\mu)}\textrm{ for $\mu_\omega \otimes \mu_\omega$-almost every $(x,y) \in \E_\omega\times \E_\omega$.}
\]
\end{theorem}
In Section \ref{exbernougibbs}, we will apply this result to random Bernoulli shifts. 
\begin{remark}[Infinite alphabets] One can observe in the proof of Theorem~\ref{thprinc2}, that stronger mixing assumptions for the stationary measure and the sample measures allow us to work with infinite alphabets. More precisely, if in Theorem~\ref{thprinc2}, one replaces assumptions (I-a) and (II) by

(I') (exponential $\phi$-mixing) the marginal measure $\mu$ satisfies
\[
\left|\mu(A\cap\sigma^{-g-n}B) -\mu(A)\mu(B)\right|\le \alpha(g)\mu(A);
\]
and 

(II') (fibered exponential $\phi$-mixing) for $\PP$-almost every $\omega\in\Omega$
\[
\left|\mu_\omega(A\cap\sigma^{-g-n}B) -\mu_\omega(A)\mu_{\theta^{n+g}\omega}(B)\right|\le \alpha(g)\mu(A),
\]
then the same conclusions are satisified.

\end{remark}

To deal with more general random subshifts (and in particular random Gibbs measures in Section \ref{exgibbs}) we will need a stronger mixing assumption on the base $(\Omega, \theta, \P)$ (satisfied for example for Anosov diffeomorphisms \cite{kotani}):
\begin{itemize}
\item[(III')](exponential $3$-mixing) There exists a Banach space $\mathcal{B}$ such that for all $\psi,\ \phi, \ \varphi\in \mathcal{B}$, for all $n\in\N^*$ and $m\in\N^*$,  we have
\[\left|\int_\Omega\psi.\phi\circ \theta^n.\varphi\circ \theta^{n+m}\, d\P-\int_\Omega \psi d\P\int_\Omega\phi d \P\int_\Omega\varphi d \P\right|\leq \|\psi\|_\mathcal{B}\|\phi\|_\mathcal{B}\|\varphi\|_\mathcal{B}\rho(\min(n,m))\]
with $\rho(n) =\mathcal{O}(a^{n})$ and $\|.\|_\mathcal{B}$ is the norm in the Banach space $\mathcal{B}$.
\end{itemize}
We are now able to replace assumption (IV) by a less restrictive assumption:
\begin{itemize}
\item[(IV')] There exists $\xi \geq0$ such that for every $n\in \N$ and every cylinder $C\in\F_0^n(X)$, the functions $\psi_1:\omega\mapsto \mu_\omega(C)$ and $\psi_2:\omega\mapsto\max_{C_n} \mu_\omega(C_n)$ (where the max is taken over all n-cylinders) belong to the Banach space $\mathcal{B}$ and
\[\|\psi_1\|_\mathcal{B}\leq \xi^n\qquad \textrm{and}\qquad \|\psi_2\|_\mathcal{B}\leq \xi^n.\]
\end{itemize}
Morever, if the base $(\Omega, \theta, \P)$ satisfies exponential $4$-mixing, it will allow us to weaken our mixing assumption for the marginal measure $\mu$ and use assumption (I):
\begin{itemize}
\item[(III'')](exponential $4$-mixing) There exists a Banach space $\mathcal{B}$ such that for all $\psi,\ \phi, \ \varphi,\ \upsilon\in \mathcal{B}$, for all $n\in\N^*$, $m\in\N^*$ and $l\in\N^*$,  we have
\[\hspace{-1cm}\left|\int_\Omega\psi.\phi\circ \theta^n.\varphi\circ \theta^{n+m}.\upsilon\circ \theta^{n+m+l}\, d\P-\int_\Omega \psi d\P\int_\Omega\phi d \P\int_\Omega\varphi d \P\int_\Omega\upsilon d \P\right|\leq \|\psi\|_\mathcal{B}\|\phi\|_\mathcal{B}\|\varphi\|_\mathcal{B}\|\upsilon\|_\mathcal{B}\rho(\min(n,m,l))\]
with $\rho(n) =\mathcal{O}(a^{n})$ and $\|.\|_\mathcal{B}$ is the norm in the Banach space $\mathcal{B}$.
\end{itemize}
In Section \ref{exgibbs}, we will check these assumptions for random Gibbs measures and will chose the Banach space $\mathcal{B}$ to be the space of H\"older continuous functions. 

With these assumptions, we obtain the same results as in Theorem~\ref{thprinc2}:
\begin{theorem} \label{thprinc3}
 If $0<\underline{H}_2(\mu)\leq\overline{H}_2(\mu)< 2h_0$ and if 
 
 $\bullet$ hypothesis (I-a), (II), (III') and (IV') are satisfied,\\
 or
 
 $\bullet$ hypothesis (I), (II), (III'') and (IV') are satisfied,\\
 then the conclusions of Theorem~\ref{thprinc2} hold.\end{theorem}

We will now apply our results to random Bernoulli shifts and random Gibbs measures (these examples follow \cite{RSV,RT}, where assumptions (I-a) and (II) where proved to obtain a quenched exponential distribution of hitting times).

\section{Examples}\label{example}
\subsection{Random Bernoulli shifts}\label{exbernougibbs}

Let $s\ge1$ and $(\Omega,\theta)$ be a subshift of finite type on the symbolic space $\{0,1,\ldots,s\}^\ZZ$ and let $\PP$ be a Gibbs measure from a H\"older potential.

Let $b\ge1$ and make the shift $\{0,1,\ldots,b\}^\NN$ a random subshift by putting on it the random Bernoulli measures
constructed as follows. Let $W=(w_{ij})$ be a $s\times b$ stochastic matrix with entries in $(0,1)$. 
Set $p_j(\omega)=w_{\omega_0,j}$. The random Bernoulli measure 
$\mu_\omega$ is defined by 
$$\mu_\omega([x_0\dots x_n])=p_{x_0}(\omega)p_{x_1}(\theta\omega)\dots p_{x_n}(\theta^n\omega).$$ 

First of all, hypothesis (IV) is satisfied since $\mu_\omega([x_0\dots x_n])$ only depends on $\omega_0,\dots,\omega_{n}$.

Since $\mu_\omega$ are Bernoulli measures, 
one can observe that for all $m,n$, $A\in\F_0^n$ and $B\in \F_0^{m}$:
\begin{equation*}
\left|\mu_\omega(A\cap\sigma^{-g-n}B) -\mu_\omega(A)\mu_{\theta^{n+g}\omega}(B)\right|=0
\end{equation*}
for every $g\geq 1$ and every $\omega\in\Omega$. Thus, property (I-a) is satisfied. 

Moreover, it was proved in \cite{RSV} that assumption (II) is satisfied. Since the Gibbs measure $\PP$ is exponentially $\psi$-mixing, it is exponentially $\rho$-mixing and (III) is satisfied. Thus, if $0<\underline{H}_2(\mu)\leq 2h_0$ one can apply Theorem~\ref{thprinc} and if besides that $\overline{H}_2(\mu)<2h_0$ then one can apply Theorem~\ref{thprinc2}.

For example, when the base is i.i.d., we can compute the R\'enyi entropy. Indeed
\begin{eqnarray*}
\mu([x_0\dots x_n])&=&\int\mu_\omega([x_0\dots x_n])d\P(\omega)\\
&=&\int p_{x_0}(\omega)p_{x_1}(\theta\omega)\dots p_{x_n}(\theta^n\omega)d\P(\omega)\\
&=&\int p_{x_0}(\omega)d\P(\omega)\int p_{x_1}(\theta\omega)d\P(\omega)\dots \int p_{x_n}(\theta^n\omega)d\P(\omega)\\
&=&\int p_{x_0}(\omega)d\P(\omega)\int p_{x_1}(\omega)d\P(\omega)\dots \int p_{x_n}(\omega)d\P(\omega).
\end{eqnarray*}
Thus,
\[\sum_{C_n}\mu(C_n)^2=\left(\sum_{x_0}\left(\int p_{x_0}(\omega)d\P\right)^2\right)^{n}\]
and
\[H_2(\mu)=-\log\left(\sum_{x_0}\left(\int p_{x_0}(\omega)d\P\right)^2\right).\]
A similar computation gives us
\[h_0=-\log\left(\int \max_{x_0}p_{x_0}(\omega)d\P\right).\]
So, if $H_2(\mu)<2h_0$, we have for $\PP$-almost every $\omega \in \Omega$,
\[
\lim_{n \to \infty}\frac{M_n(x,y)}{\log n} = \frac{2}{-\log\left(\sum_{x_0}\left(\int p_{x_0}(\omega)d\P\right)^2\right)}\]
for $\mu_\omega \otimes \mu_\omega$-almost every $(x,y) \in X\times X$.

In this case, wether the condition $H_2(\mu)<2h_0$ holds or not can be easily checked. For example, this condition will be satisfied if the letter with the maximum weight is always the same. Indeed, assuming that it exists $b_0\in\{0,\dots,b\}$ such that $\max_{x_0}p_{x_0}(\omega)=p_{b_0}$ for every $\omega\in\Omega$, we observe that
\[
\sum_{x_0}\left(\int p_{x_0}(\omega)d\P\right)^2=p_{b_0}^2+\sum_{x_0\neq b_0}\left(\int p_{x_0}(\omega)d\P\right)^2>p_{b_0}^2
\]
and thus
\[H_2(\mu)<-\log (p_{b_0}^2)=2h_0.\]
Also, the condition $H_2(\mu)<2h_0$ will be satisfied if all the letters have a relatively close probability, i.e. if it exists a constant $P$ such that $P<p_j(\omega)<P\sqrt{b+1}$ for every $j\in\{0,\dots,b\}$ and every $\omega\in\Omega$. Indeed, in this case, we have
\begin{eqnarray*}
\left(\int \max_{x_0}p_{x_0}(\omega)d\P\right)^2&<&\left(\int P\sqrt{b+1}d\P\right)^2=P^2(b+1)=\sum_{x_0}\left(\int P d\P\right)^2\\
&<&\sum_{x_0}\left(\int p_{x_0}(\omega)d\P\right)^2
\end{eqnarray*}
and thus $H_2(\mu)<2h_0$. This could be applied to small perturbations of a uniform Bernoulli shift, i.e., $p_j(\omega)= \frac{1}{b+1}+\delta_j(\omega)$ with $\frac{1-\sqrt{b+1}}{(b+1)(1+\sqrt{b+1})}<\delta_j(\omega)<\frac{\sqrt{b+1}-1}{(b+1)(1+\sqrt{b+1})}$ for every $j\in\{0,\dots,b\}$ and every $\omega\in\Omega$ (one can easily check that in this case $P<p_j(\omega)<P\sqrt{b+1}$ with $P=\frac{2}{(b+1)(1+\sqrt{b+1})}$).

\subsection{Random Gibbs measures}\label{exgibbs}

In this section we will give details on a family of shifts which satisfy our assumptions. 
 
We will use the approach detailed in \cite{stad} which is concerned with shifts on $\N$, for example the full shift.  We note that this extends a little beyond the full shift, to the so-called BIP setting.  

We assume that $(\Omega, \p, \theta)$ is an invertible measure preserving system and let $X=N^{\N}$ and let $\sigma:X\to X$ denote the shift. For $r\in (0,1)$, let $d_r$ be the usual symbolic metric on $X$, i.e., $d_r(x, y)=r^k$ where $x_i=y_i$ for $i=0, \ldots, k-1$, but $x_k\neq y_k$.

Assume that $\phi:X\times \Omega:\to \R$ is a function which is almost surely H\"older continuous, which is to say, for $$V_n^\omega(\phi):=\sup\{|\phi_\omega(x)-\phi_\omega(y)|:x_i=y_i,\ i=0,\ldots, n-1\},$$
there is some $r\in (0,1)$ and $\kappa(\omega)\ge 0$ such that $\int\log\kappa~d\p<\infty$ where $V_n^\omega(\phi)\le \kappa(\omega)r^n$.

Define $S_n\phi_\omega(x):=\sum_{k=0}^{n-1}\phi_{\theta^{k}\omega}\circ\sigma^k(x)$.  If $x,y$ are in the same $m$-cylinder for $m\ge n$, then $|S_n\phi_\omega(x)-S_n\phi_\omega(y)|\le r^{m-n}\sum_{k=0}^{n-1}r^k\kappa(\theta^{n-k}\omega)$.  As in the proof of  \cite[Lemma 7.2]{DenKifSta08-2}, the assumption on the integrability of $\log\kappa$ implies that the above limit is finite a.s., say $\sum_{k=0}^{n-1}r^k\kappa(\theta^{n-k}\omega) \le c_\omega$.  However, it is also pointed out in \cite{stad} that if $\kappa$ is integrable, then we have an a.s. uniform upper bound, say $C_\phi$ on $\sum_{k=0}^{n-1}r^k\kappa(\theta^{n-k}\omega)$.  Given a H\"older function $\psi$, then we define
$$D_\omega(\psi):=\sup_{x,y\in X_\omega}\left\{\frac{|\psi(x)-\psi(y)|}{d_r(x, y)}\right\}.$$

Now we define the random Ruelle operator by
$$\L_\omega\psi(x)=\sum_{\sigma y=x}e^{\phi_\omega(y)}\psi(y)$$
where $\psi:X'\to [0,\infty]$ where $X'\subset X$ is such that $\L_\omega$ is well-defined. As in \cite{DenKifSta08, stad}, it can be shown that there exists some constant $\lambda_\omega$ and some measurable function $\rho_\omega$ which is uniformly bounded from below, such that  $\L_\omega\rho_\omega=\lambda_\omega\rho_{\theta\omega}$ a.s. and such that $\log\rho$ satisfies the same smoothness properties as $\phi$, i.e. we have the same $\kappa$ and $r$ in the variation.  This allows us to replace $\phi$ with
$$\varphi_\omega(x) :=\phi_\omega(x)+\log \rho_\omega-\log\rho_{\theta\omega}(\sigma x)-\log\lambda_\omega.$$
Letting $\L_\omega$ denote the corresponding transfer operator, one consequence of this is that $\L_\omega 1=1$.  Note also that random equilibrium states for $\phi$ and $\varphi$ coincide.

Now we have the property that
\begin{equation}
\int\L_\omega^n(\psi)\cdot \gamma~d\mu_\omega=\int\psi\cdot \gamma\circ \sigma^n~d\mu_{\theta^{-n}\omega}\label{eq:L basic}
 \end{equation}
for appropriate observables $\psi, \gamma$.

We will make the following almost sure assumptions on our system (which are satisfied for subshifts of finite type with H\"older potentials):

\begin{enumerate}

\item $\int\kappa~d\p<\infty$, so $\sum_{k=0}^\infty r^k\kappa(\theta^{n-k}\omega)$ is a.s. uniformly bounded, independently of $\omega$.

\item There exists a measure $\mu_\omega$ where $\L_\omega^*\mu_\omega=\mu_{\theta^{-1}\omega}$, i.e., \eqref{eq:L basic} holds for $L^1$ observables. 

\item Big images:  there exists some $C_{BIP}>0$ such that for any $n$-cylinder $U$ and $\omega\in \Omega$, $\inf {\mu_{\theta^{n}\omega}(\sigma^{n}U)}>C_{BIP}$.

\item There exist $C>0$, and $g(n)\to 0$ as $n\to \infty$ such that 
$$\left\|\L_\omega^n(\psi)-\int\psi~d\mu_{\omega}\right\|_\infty\le Cg(n)D_\omega(\psi).$$
\end{enumerate}
 
Under these conditions, it was proved in \cite[Proposition 6.1]{RT} that the sample measures satisfy (II). 

When $\theta:\Omega\to \Omega$ is a subshift of finite type on a finite alphabet, with a Gibbs measure for a H\"older potential, it is known that assumption (III'') is satisfied with $\mathcal{B}$ being the space of H\"older continuous functions \cite{kotani, ruhr}. 

For $\alpha>0$, let the norm $\|\cdot\|_\alpha$ be defined by $\|\cdot\|_\alpha=|\cdot|_\alpha+|\cdot|_\infty$ where $|f|_\alpha=\sup\left\{\frac{V_n(f)}{\alpha^n}:n\ge 0\right\}$.

It was also proved in \cite[Lemma 6.2]{RT} that for any $\beta\in (0,1]$, there exist $\alpha\in (0,1]$ and $C_\beta>0$, such that for every cylinder $C$ in $\mathcal{F}_0^n(X)$, the map $\psi_1:\omega\mapsto \mu_\omega(C)$ is $\alpha$-H\"older and $\|\psi_1\|_\alpha\leq C_\beta r^{-\beta n}$. Thus, $\|\psi_1\|_\alpha\leq\xi^n$ for some $\xi\geq0$. Moreover, since for every real-valued functions $f,g$ we have $|\max f(x)-\max g(x)|\leq \max|f(x)-g(x)|$, we obtain that the map $\psi_2:\omega\mapsto \max_{C_n}\mu_\omega(C_n)$ is $\alpha$-H\"older and $\|\psi_2\|_\alpha\leq \xi^n$. Thus, (IV') is satisfied.

Assumption (I-a) has been proved in \cite[Section 6.2]{RT}. However, our proof contains a mistake since both terms in the right-hand side of the first equation in page 149 should be with the H\"older norm. In fact, we will prove that the sample measures satisfy (I). One can observe that to obtain Theorem~2.2 in \cite{RT}, (I-a) was only used in equation (4.3) and could be substituted by (I).

Following the proof of \cite[Proposition 2.4]{ParPol90}, we fix our set $A\in\F_0^n$ and take both $w$ and $v$ to be $w(\omega)=w_A(\omega)=\mu_\omega(A)-\mu(A)$ (this normalisation by $\mu(A)$ simplifies the calculations). Note that $|w|_\infty \le 1$.  Let $\gamma\in(0,\alpha)$. For $k\in \N$, we approximate $v, w$ by $v_k, w_k$, depending only on coordinates $x_{-k}, \ldots, x_0, \ldots, x_k$ such that $|v-v_k|\le \gamma^k |v|_\gamma$ and $|w-w_k|\le \gamma^k |w|_\gamma$.   So that proof yields that for $n\ge 2k$,

\begin{align*}
\left|\int v\circ \theta^{n+g} w~d\p\right| &\le  \left|\int (v-v_k)\circ \theta^{n+g} w~d\p\right|+ \left|\int v_k\circ \theta^{n+g} (w-w_k)~d\p\right|\\
& \hspace{5cm} + \left|\int v_k\circ \theta^{n+g}w_k~d\p\right| \\
& \le \gamma^k|v|_\gamma|w|_\infty+\gamma^k|w|_\gamma|v|_\infty +K'\rho^{n+g-2k}|v|_\infty|w_k|_\infty\\
& \le 2\gamma^k C_\beta r^{-\beta n} +K'\rho^{n+g-2k}.
\end{align*}

So taking $k=\lfloor \frac{n+g}3\rfloor$, if we choose $\gamma\in(0,\alpha)$ so that $\gamma^{\frac13}r^{-\beta}<1$, we obtain
\begin{equation}\label{gibbssamp1}
\left|\int v\circ \theta^{n+g} w~d\p\right| \le 2 C_\beta\gamma^{\frac g3}+ K'\rho^{\frac{n+g}3}. 
\end{equation}
Moreover, we can observe that by (II)
\begin{eqnarray}
\left|\mu(A\cap\sigma^{-g-n}A) -\mu(A)^2\right| &=&\left|\int\mu_\omega(A\cap\sigma^{-g-n}A)d\P-\mu(A)^2\right|\nonumber\\ 
&\leq&\left|\int\mu_\omega(A)\mu_{\theta^{g+n}\omega}(A)d\P-\mu(A)^2\right|+\alpha(g)\nonumber\\ 
&=& \alpha(g)+\left|\int v\circ \sigma^{n+g} w~d\p\right|.\label{gibbssamp2}
\end{eqnarray}
Thus, by \eqref{gibbssamp1} and \eqref{gibbssamp2}, (I) is verified.

Finally, we showed that if the fiber maps satisfy conditions 1.--4. and the base transformation is a subshift of finite type on a finite alphabet with a Gibbs measure for some H\"older potential, then assumptions (I), (II), (III'') and (IV') are satisfied. Thus, if  $0<\underline{H}_2(\mu)\leq\overline{H}_2(\mu)< 2h_0$, one can apply Theorem~\ref{thprinc3}.



\section{Proofs}\label{secproof}
In this section, we will prove our theorems. Both proofs follow the line of \cite{BaLiRo} but diverge at some point since the samples measures are not invariant but satisfy \eqref{eqinv}.
\begin{proof}[Proof of Theorem \ref{thprinc}]
For simplicity we assume $\alpha(g)=e^{-g}.$ 
Let $\eps >0$ and define
$$k_n=\left\lceil\frac{2 \log n + d\log \log n}{\underline{H_2}(\mu)-\eps}\right\rceil$$
where $d$ is a constant to be chosen later.

Let us also denote
$$A_{i,j}(y)=\sigma^{-i}[C_{k_n}(\sigma^jy)]$$
and
$$S_n(x,y)= \sum_{i,j= 0, \ldots, n-1}\mathbbm{1}_{A_{i,j}(y)}(x).$$

Let $\omega\in\Omega$ such that \eqref{eqinv} is satisfied. Using Markov's inequality we obtain
\begin{equation}\label{mark}
\mu_\omega \otimes \mu_\omega \left(\left\{(x,y): M_n(x,y) \geq k_n\right\}\right)=\mu_\omega \otimes \mu_\omega \left(\left\{(x,y): S_n(x,y)\geq1\right\}\right)  \leq  \EE_\omega\left(S_n\right).
\end{equation}
Moreover, the invariance formula \eqref{eqinv} of the sample measures gives us
\begin{eqnarray*}
\EE_\omega\left(S_n\right) & = & \int \int \displaystyle\sum_{i,j=0,\ldots,n-1} \mathbbm{1}_{C_{k_n}(\sigma^j y)}(\sigma^i x) \ d\mu_\omega(x) \ d\mu_\omega(y) \\
                                                 & = & \displaystyle\sum_{i,j=0,\ldots,n-1} \int \int \mathbbm{1}_{C_{k_n}(\sigma^j y)}(x) \ d\mu_{\theta^i\omega}(x)  \ d\mu_\omega(y) \\
                                                 & = & \displaystyle\sum_{i,j=0,\ldots,n-1} \int \mu_{\theta^i\omega}\left(C_{k_n}(\sigma^j y)\right) \ d\mu_\omega(y) \\
                                               & = &  \displaystyle\sum_{i,j=0,\ldots,n-1} \int \mu_{\theta^i\omega}\left(C_{k_n}(y)\right) \ d\mu_{\theta^j\omega}(y) \\
                                                                                                  & =& \displaystyle\sum_{i,j=0,\ldots,n-1} \sum_{C_{k_n}}\mu_{\theta^i\omega}\left(C_{k_n}\right)\mu_{\theta^j\omega}\left(C_{k_n}\right).
\end{eqnarray*}

One can notice that, since the sample measures are not invariant, we cannot estimate the previous sum directly as in the deterministic case \cite{BaLiRo}. Thus, this is where our proof will differ and where we will use the mixing assumptions which where not necessary in the deterministic proof. First of all, using Markov's inequality, we observe that

\[
\PP\left(\omega: \EE_\omega(S_n)\geq \frac{1}{\log n}\right)\leq \log n . \int\displaystyle\sum_{i,j=0,\ldots,n-1} \sum_{C_{k_n}}\mu_{\theta^i\omega}\left(C_{k_n}\right)\mu_{\theta^j\omega}\left(C_{k_n}\right) d\PP(\omega)
\]
To study the behaviour of the integral on the right hand side of the previous inequality, we will divide the sum in two terms, when $i$ and $j$ are far from one another and when they are not. Let us define $m=\gamma \log n$ where $\gamma >0$ will be chosen later.

When $i$ and $j$ are close from one another, we have, using that $\mu_{\theta^j\omega}$ is a probability measure and the invariance of $\PP$ 
\begin{eqnarray*}
\sum_{|i-j|\leq m} \sum_{C_{k_n}}\int\mu_{\theta^i\omega}\left(C_{k_n}\right)\mu_{\theta^j\omega}\left(C_{k_n}\right) d\PP(\omega) &\leq & \sum_{|i-j|\leq m} \int\max_{C_{k_n}}\mu_{\theta^i\omega}\left(C_{k_n}\right).\sum_{C_{k_n}}\mu_{\theta^j\omega}\left(C_{k_n}\right) d\PP(\omega) \\
&= & \sum_{|i-j|\leq m} \int\max_{C_{k_n}}\mu_{\theta^i\omega}\left(C_{k_n}\right)d\PP(\omega)\\
&\leq & 2mn \int\max_{C_{k_n}}\mu_{\omega}\left(C_{k_n}\right)d\PP(\omega).  
\end{eqnarray*}
When $i$ and $j$ are far from one another, we can use the mixing assumptions (I) and (II) to obtain
\begin{eqnarray*}
& &\sum_{|i-j|> m} \sum_{C_{k_n}}\int\mu_{\theta^i\omega}\left(C_{k_n}\right)\mu_{\theta^j\omega}\left(C_{k_n}\right) d\PP(\omega) \\
&\leq &2\sum_{j> i+m} \sum_{C_{k_n}}\left(\alpha(m-k_n)+\int\mu_{\theta^i\omega}\left(C_{k_n}\cap \sigma^{-(j-i)}C_{k_n}\right) d\PP(\omega)\right)\\
&\leq &n^2 N^{k_n}\alpha(m-k_n)+2\sum_{j> i+m} \sum_{C_{k_n}}\mu\left(C_{k_n}\cap \sigma^{-(j-i)}C_{k_n}\right)\\
&\leq &2n^2 N^{k_n}\alpha(m-k_n)+n^2\sum_{C_{k_n}}\mu\left(C_{k_n}\right)^2.
\end{eqnarray*}
Thus, we obtain, for $n$ large enough,
\begin{eqnarray*}
\PP\left(\omega: \EE_\omega(S_n)\geq \frac{1}{\log n}\right)&\leq& \log n \left(2mn \int\max_{C_{k_n}}\mu_{\omega}\left(C_{k_n}\right)d\PP(\omega)+2n^2 N^{k_n}\alpha(m-k_n)+n^2\sum_{C_{k_n}}\mu\left(C_{k_n}\right)^2\right)\\
&\leq&\log n \left(2mn e^{-k_n(h_0-\eps/2)}+2n^2 N^{k_n}e^{-(m-k_n)}+n^2e^{-k_n(\underline{H}_2(\mu)-\eps)}\right)
\end{eqnarray*}
where the last inequality came from the definition of $h_0$ and $\underline{H}_2(\mu)$.

Then, choosing $d>0$ large enough and $\gamma>0$ large enough, we have, by definition of $k_n$ and since $\underline{H}_2(\mu)\leq 2 h_0$, that
\[\PP\left(\omega: \EE_\omega(S_n)\geq \frac{1}{\log n}\right)\leq \mathcal{O}\left(\frac{1}{\log n}\right).\]

Choosing a subsequence $\{n_{\kappa}\}_{\kappa \in \mathbb{N}}$ such that $n_{\kappa}= \lceil e^{\kappa^2}\rceil$ we have that
$$
\PP\left(\omega: \EE_\omega(S_{n_\kappa})\geq \frac{1}{\log n_\kappa}\right)\leq \frac{1}{\kappa^2} \ .
$$

Since the last quantity is summable in $\kappa$, the Borel-Cantelli lemma gives that for $\PP$-almost every $\omega\in\Omega$, if $\kappa$ is large enough then
 \[\EE_\omega(S_{n_\kappa})\leq \frac{1}{\log n_\kappa}=\frac{1}{\kappa^2}.\]
 Thus, this inequality together with \eqref{mark} gives us that for $\PP$-almost every $\omega\in\Omega$, if $\kappa$ is large enough then
 \[\mu_\omega \otimes \mu_\omega \left(\left\{(x,y): M_{n_\kappa}(x,y) \geq k_{n_\kappa}\right\}\right)\leq\frac{1}{\kappa^2}.\]
 As previously, since the last quantity is summable in $\kappa$, the Borel-Cantelli lemma gives that for $\mu_\omega\otimes \mu_\omega$-almost every $(x,y)$, if $\kappa$ is large enough then
$$M_{n_{\kappa}}(x,y) < k_{n_{\kappa}}$$
and  then
\begin{eqnarray*}
\frac{ M_{n_{\kappa}}(x,y)}{\log n_{\kappa}} \leq  \frac{1}{\underline{H}_2(\mu)-\eps}\left( 2 + \frac{\log \log n_{\kappa}}{\log n_{\kappa}}\right).
\end{eqnarray*}

Finally, taking the limit superior in the previous equation and observing that  $(n_\kappa)_\kappa$ is increasing, $(M_n)_n$ is increasing and $\underset{\kappa\rightarrow+\infty}\lim\frac{\log n_\kappa}{\log n_{\kappa+1}}=1$, we have for $\mu_\omega\otimes \mu_\omega$-almost every $(x,y)$
\[ \underset{n\rightarrow+\infty}{\overline\lim}\frac{M_n(x,y)}{\log n}=\underset{\kappa\rightarrow+\infty}{\overline\lim}\frac{\log M_{n_\kappa}(x,y)}{\log n_\kappa}\leq \frac{2}{\underline{H}_2(\mu)-\eps}.\]
Then the theorem is proved since $\eps$ can be chosen arbitrarily small.

\end{proof}


\begin{proof}[Proof of Theorem \ref{thprinc2} and Theorem \ref{thprinc3} ]
For $\eps>0$, let us define 
\[k_n=\left\lfloor\frac{1}{\overline{H}_2(\mu)+\eps}(2\log n+d\log\log n)\right\rfloor\]
where $d$ is a constant that we will choose later.

Let $\omega\in\Omega$ such that \eqref{eqinv} is satisfied. As in the proof of Theorem \ref{thprinc}, we have
\begin{equation*}
\EE_\omega(S_n)=\sum_{i,j=0,\ldots,n-1} \sum_{C_{k_n}}\mu_{\theta^i\omega}\left(C_{k_n}\right)\mu_{\theta^j\omega}\left(C_{k_n}\right).\end{equation*}

Following the lines of the proof of Theorem 7 in \cite{BaLiRo}, we have, by Chebyshev's inequality,
\begin{equation*}
\mu_\omega\otimes \mu_\omega \left((x,y): M_n(x,y)< k_n\right)=\mu_\omega\otimes \mu_\omega \left((x,y): S_n(x,y)=0\right)\leq\frac{\var_\omega(S_n)}{\EE_\omega(S_n)^2}.\end{equation*}
Thus, we need to control the variance of $S_n$. First of all, we observe that
\[
\var_\omega(S_n)
=\sum_{0\leq i,i',j,j'\leq n-1}\iint \mathbbm{1}_{C_{k_n}(\sigma^jy)}(\sigma^ix)\mathbbm{1}_{C_{k_n}(\sigma^{j'}y)}(\sigma^{i'}x)d\mu_\omega(x)d\mu_\omega(y)-\EE_\omega(S_n)^2.
\]

We will estimate the variance dividing the sum of $\var(S_n)$ into $4$ terms. Let $g=\log(n^\beta)$ where $\beta$ is a constant that we will choose later. 

For $i'-i>g+k_n$, we use the invariance formula \eqref{eqinv} and the mixing assumption (II) to obtain:
\begin{eqnarray*}
& &\iint \mathbbm{1}_{C_{k_n}(\sigma^jy)}(\sigma^ix)\mathbbm{1}_{C_{k_n}(\sigma^{j'}y)}(\sigma^{i'}x)d\mu_\omega(x)d\mu_\omega(y)\\
&=&\iint \mathbbm{1}_{C_{k_n}(\sigma^jy)}(x)\mathbbm{1}_{C_{k_n}(\sigma^{j'}y)}(\sigma^{i'-i}x)d\mu_{\theta^i\omega}(x)d\mu_\omega(y)\\
&\leq&\alpha(g)+\int\mu_{\theta^i\omega}\left(C_{k_n}(\sigma^{j}y)\right)\mu_{\theta^{i'}\omega}\left(C_{k_n}(\sigma^{j'}y)\right)d\mu_\omega(y).
\end{eqnarray*}
If, moreover, $j'-j>g+k_n$, using again the mixing assumption (II), we have
\begin{eqnarray*}
& &\int\mu_{\theta^i\omega}\left(C_{k_n}(\sigma^{j}y)\right)\mu_{\theta^{i'}\omega}\left(C_{k_n}(\sigma^{j'}y)\right)d\mu_\omega(y)\\
&=& \int\mu_{\theta^i\omega}\left(C_{k_n}(y)\right)\mu_{\theta^{i'}\omega}\left(C_{k_n}(\sigma^{j'-j}y)\right)d\mu_{\theta^j\omega}(y)\\
 &=&\sum_{C_{k_n},C'_{k_n}}\mu_{\theta^i\omega}(C_{k_n})\mu_{\theta^{i'}\omega}(C'_{k_n})\mu_{\theta^j\omega}\left(C_{k_n}\cap \sigma^{-(j'-j)}C'_{k_n}\right)\nonumber\\
 &\leq&\alpha(g)+\sum_{C_{k_n},C'_{k_n}}\mu_{\theta^i\omega}(C_{k_n})\mu_{\theta^{i'}\omega}(C'_{k_n})\mu_{\theta^j\omega}(C_{k_n})\mu_{\theta^{j'}\omega}(C'_{k_n}).
\end{eqnarray*}
However, if $j'-j\leq g+k_n$, we obtain:
\[\int\mu_{\theta^i\omega}\left(C_{k_n}(\sigma^{j}y)\right)\mu_{\theta^{i'}\omega}\left(C_{k_n}(\sigma^{j'}y)\right)d\mu_\omega(y)\leq\max_{C_{k_n}}\mu_{\theta^{i'}\omega}\left(C_{k_n}\right)\int\mu_{\theta^i\omega}\left(C_{k_n}(y)\right)d\mu_{\theta^j\omega}(y).\]
By symmetry, the case where $i'-i\leq g+k_n$ and $j'-j>g+k_n$ will be treated as the previous one.

Finally, when $|i-i'|\leq g+k_n$ and $|j-j'|\leq g+k_n$, we have:
\begin{eqnarray*}
\iint \mathbbm{1}_{C_{k_n}(\sigma^jy)}(\sigma^ix)\mathbbm{1}_{C_{k_n}(\sigma^{j'}y)}(\sigma^{i'}x)d\mu_\omega(x)d\mu_\omega(y)&\leq &\iint \mathbbm{1}_{C_{k_n}(\sigma^jy)}(\sigma^ix)d\mu_\omega(x)d\mu_\omega(y)\\
&=&\int\mu_{\theta^i\omega}\left(C_{k_n}(y)\right)d\mu_{\theta^j\omega}(y).
\end{eqnarray*}

Then, one can gather these estimates to obtain
\begin{eqnarray}
& &\mu_\omega\otimes \mu_\omega \left((x,y): M_n(x,y)< k_n\right)\nonumber\\
& \leq &\frac{2n^4\alpha(g)+2(g+k_n)\underset{0\leq i,j,i'\leq n-1}\sum\underset{C_{k_n}}\max\,\mu_{\theta^{i'}\omega}\left(C_{k_n}\right)\int\mu_{\theta^i\omega}\left(C_{k_n}(y)\right)d\mu_{\theta^j\omega}(y)}{\EE_\omega(S_n)^2}\nonumber\\
& &+\frac{(g+k_n)^2\underset{0\leq i,j\leq n-1}\sum\int\mu_{\theta^i\omega}\left(C_{k_n}(y)\right)d\mu_{\theta^j\omega}(y)}{\EE_\omega(S_n)^2}\nonumber\\
&=&\frac{2n^4\alpha(g)}{\EE_\omega(S_n)^2}+\frac{2(g+k_n)\underset{0\leq i'\leq n-1}\sum\underset{C_{k_n}}\max\,\mu_{\theta^{i'}\omega}\left(C_{k_n}\right)+(g+k_n)^2}{\EE_\omega(S_n)}.\label{ineqvar}
\end{eqnarray}
This is where the proof diverge completely from the deterministic case. Indeed, as in the proof of Theorem \ref{thprinc}, we cannot treat directly the previous estimate (which was possible in the deterministic case) and an extra care is needed. To deal with the term with the maximum, we use Markov's inequality to obtain
\begin{eqnarray}
\P\left(\underset{0\leq i'\leq n-1}\sum\underset{C_{k_n}}\max\,\mu_{\theta^{i'}\omega}\left(C_{k_n}\right)\geq1 \right)&\leq&\int_\Omega\underset{0\leq i'\leq n-1}\sum\underset{C_{k_n}}\max\,\mu_{\theta^{i'}\omega}\left(C_{k_n}\right)d\P(\omega)\nonumber\\
&=&n\int_\Omega\underset{C_{k_n}}\max\,\mu_{\omega}\left(C_{k_n}\right)d\P(\omega)\nonumber\\
&\leq&ne^{k_n(h_0-\eps)}.\label{estmax}
\end{eqnarray}
Since $\overline{H}_2(\mu)<2h_0$, one can choose $\eps$ small enough such that $ne^{k_n(h_0-\eps)}\leq n^{-\eps}$ for every $n$ large enough.

To deal with the expectation in the denominator in \eqref{ineqvar}, we will need the following lemma (which proof can be found after the proof of the theorem).

\begin{lemma}\label{lemesp}
Let $\frac{3}{4}<\delta<1$. Under the assumptions of Theorem \ref{thprinc2} or Theorem \ref{thprinc3}, we have
\[\P\left(\bigg|\EE_\omega(S_n)-n^2\sum_{C_{k_n}}\mu\left(C_{k_n}\right)^2\bigg|\geq\bigg(n^2\sum_{C_{k_n}}\mu(C_{k_n})^2\bigg)^\delta\right)=\mathcal{O}\left(\frac{1}{\log n}\right).\]
\end{lemma}
Thus, using this lemma with \eqref{estmax}, we have
\[\P(|\EE_\omega(S_n)-n^2\sum_{C_{k_n}}\mu(C_{k_n})^2|\geq(n^2\sum_{C_{k_n}}\mu(C_{k_n})^2)^\delta\bigcup\underset{0\leq i'\leq n-1}\sum\underset{C_{k_n}}\max\,\mu_{\theta^{i'}\omega}\left(C_{k_n}\right)\geq1 )=\mathcal{O}\left(\frac{1}{\log n}\right).\]
Choosing a subsequence $\{n_{\kappa}\}_{\kappa \in \mathbb{N}}$ such that $n_{\kappa}= \lceil e^{\kappa^2}\rceil$, the Borel-Cantelli lemma gives that for $\PP$-almost every $\omega\in\Omega$, if $\kappa$ is large enough then
\begin{equation}\label{eqmax}
\underset{0\leq i'\leq n_\kappa-1}\sum\underset{C_{k_{n_\kappa}}}\max\,\mu_{\theta^{i'}\omega}\left(C_{k_{n_\kappa}}\right)\leq1 
\end{equation}
and
\[|\EE_\omega(S_{n_\kappa})-n_\kappa^2\sum_{C_{k_{n_\kappa}}}\mu(C_{k_{n_\kappa}})^2|\leq(n_\kappa^2\sum_{C_{k_{n_\kappa}}}\mu(C_{k_{n_\kappa}})^2)^\delta.\]
Thus, if $\kappa$ is large enough
\begin{equation}\label{eqespkappa}
\EE_\omega(S_{n_\kappa})\geq n_\kappa^2\sum_{C_{k_{n_\kappa}}}\mu(C_{k_{n_\kappa}})^2-(n_\kappa^2\sum_{C_{k_{n_\kappa}}}\mu(C_{k_{n_\kappa}})^2)^\delta\geq(n_\kappa^2\sum_{C_{k_{n_\kappa}}}\mu(C_{k_{n_\kappa}})^2)^\delta.
\end{equation}
 Thus, \eqref{ineqvar} together with \eqref{eqmax} and \eqref{eqespkappa} gives us that for $\PP$-almost every $\omega\in\Omega$, if $\kappa$ is large enough then
\begin{eqnarray*}
\mu_\omega\otimes \mu_\omega \left((x,y): M_{n_\kappa}(x,y)< k_{n_\kappa}\right)&\leq& \frac{2{n_\kappa}^4\alpha(g)}{({n_\kappa}^2\sum_{C_{k_{n_\kappa}}}\mu(C_{k_{n_\kappa}})^2)^{2\delta}}+\frac{2(g+k_{n_\kappa})+(g+k_{n_\kappa})^2}{({n_\kappa}^2\sum_{C_{k_{n_\kappa}}}\mu(C_{k_{n_\kappa}})^2)^\delta}\\
&\leq&\frac{2{n_\kappa}^4\alpha(g)}{(\log n_\kappa)^{-2b\delta}}+\frac{2(g+k_{n_\kappa})+(g+k_{n_\kappa})^2}{(\log n_\kappa)^{-b\delta}}
\end{eqnarray*}
where the last inequality came from the definition of $\overline{H}_2(\mu)$ and our choice of $k_n$. Finally, choosing $\beta$ large enough in the definition of $g$ and choosing $d<0$ small enough, we obtain that if $\kappa$ is large enough
\[\mu_\omega\otimes \mu_\omega \left((x,y): M_{n_\kappa}(x,y)< k_{n_\kappa}\right)\leq \frac{1}{\log n_\kappa}=\frac{1}{\kappa^2}.\]
Since the last quantity is summable in $\kappa$, the Borel-Cantelli lemma gives that for $\mu_\omega\otimes \mu_\omega$-almost every $(x,y)$, if $\kappa$ is large enough then
$$M_{n_{\kappa}}(x,y) \geq k_{n_{\kappa}}$$
and  then
\begin{eqnarray*}
\frac{ M_{n_{\kappa}}(x,y)}{\log n_{\kappa}} \geq  \frac{1}{\overline{H}_2(\mu)+\eps}\left( 2 + b\frac{\log \log n_{\kappa}}{\log n_{\kappa}}\right).
\end{eqnarray*}

Finally, using the same arguments as in the proof of Theorem~\ref{thprinc}, we have for $\P$-almost every $\omega$
\[ \underset{n\rightarrow+\infty}{\underline\lim}\frac{M_n(x,y)}{\log n}=\underset{\kappa\rightarrow+\infty}{\underline\lim}\frac{\log M_{n_\kappa}(x,y)}{\log n_\kappa}\geq \frac{2}{\overline{H}_2(\mu)+\eps}\]
for $\mu_\omega\otimes \mu_\omega$-almost every $(x,y)$.

Then the theorems are proved since $\eps$ can be chosen arbitrarily small.
\end{proof}
\begin{proof}[Proof of Lemma~\ref{lemesp}]
As in the previous proof, we take $k_n=\lfloor\frac{1}{\overline{H}_2(\mu)+\eps}(2\log n+d\log\log n)\rfloor$ and $g=\log(n^\beta)$ where $d<0$ and $\beta>0$ are constants to be chosen later.

First of all, we use Markov's inequality
\begin{eqnarray}
& &\P\left(\bigg|\EE_\omega(S_n)-n^2\sum_{C_{k_n}}\mu(C_{k_n})^2\bigg|\geq(n^2\sum_{C_{k_n}}\mu(C_{k_n})^2)^\delta\right)\nonumber\\
&=&\P\left(\bigg|\EE_\omega(S_n)-n^2\sum_{C_{k_n}}\mu(C_{k_n})^2\bigg|^2\geq(n^2\sum_{C_{k_n}}\mu(C_{k_n})^2)^{2\delta}\right)\nonumber\\
&\leq&\frac{\int|\EE_\omega(S_n)-n^2\sum_{C_{k_n}}\mu(C_{k_n})^2|^2 d\P(\omega)}{(n^2\sum_{C_{k_n}}\mu(C_{k_n})^2)^{2\delta}}\nonumber\\
&=&\frac{\int\left(\EE_\omega(S_n)^2+(n^2\sum_{C_{k_n}}\mu(C_{k_n})^2)^2-2\EE_\omega(S_n)n^2\sum_{C_{k_n}}\mu(C_{k_n})^2 \right)d\P(\omega)}{(n^2\sum_{C_{k_n}}\mu(C_{k_n})^2)^{2\delta}}.\label{eqcheb}
\end{eqnarray}
Firstly, we will treat the last term on the previous numerator, using the mixing assumptions (I) and (II)
\begin{eqnarray}
\int\EE_\omega(S_n)d\P&=& \sum_{C_{k_n}}\sum_{i,j=0,\ldots,n-1}\int\mu_{\theta^i\omega}\left(C_{k_n}\right)\mu_{\theta^j\omega}\left(C_{k_n}\right)d \P\nonumber\\
&\geq&\sum_{C_{k_n}}\sum_{|i-j|\geq g+k_n} \int\mu_{\theta^i\omega}\left(C_{k_n}\right)\mu_{\theta^j\omega}\left(C_{k_n}\right)d \P\nonumber\\
&\geq&2\sum_{C_{k_n}}\sum_{j\geq i+g+k_n}\left( \int\mu_{\theta^i\omega}\left(C_{k_n}\cap\sigma^{-(j-i)}C_{k_n}\right)d \P-\alpha(g)\right)\nonumber\\
&\geq&\sum_{C_{k_n}}\sum_{|i-j|\geq g+k_n}\left( \mu\left(C_{k_n}\right)^2-2\alpha(g)\right)\nonumber\\
&\geq&n(n-(g+k_n))\left( \sum_{C_{k_n}}\mu\left(C_{k_n}\right)^2-2\alpha(g)N^{k_n}\right)\nonumber\\
&\geq&n^2\sum_{C_{k_n}}\mu\left(C_{k_n}\right)^2-n(g+k_n)\sum_{C_{k_n}}\mu\left(C_{k_n}\right)^2-2n^2\alpha(g)N^{k_n}.\label{lowerint}
\end{eqnarray}
To get an estimate on \eqref{eqcheb}, we need to study the term $\int\EE_\omega(S_n)^2 d\P$. One can observe that
\[\int\EE_\omega(S_n)^2d\P= \sum_{C_{k_n},{C}_{k_n}'}\sum_{i,j,i',j'=0,\ldots,n-1}\int\mu_{\theta^i\omega}\left(C_{k_n}\right)\mu_{\theta^j\omega}\left(C_{k_n}\right)\mu_{\theta^{i'}\omega}\left({C}_{k_n}'\right)\mu_{\theta^{j'}\omega}\left({C}_{k_n}'\right)d \P.\]
We will separate the study of this integral depending on the relative distance and position between $i,j,i'$ and $j'$ and consider 5 different cases.

Case 1: $i,j,i'$ and $j'$ are all far from one another, i.e. at least at a distance greater that $g+k_n$. We will assume that $i<j<i'<j'$ (when the relative position is different, everything can be done identically because of the symmetry) and that $j-i>g+k_n$, $i'-j>g+k_n$, $j'-i'>g+k_n$. Using the mixing assumptions (I-a) and (II) (a similar estimate is obtained when (III'') is satisfied) we obtain
\begin{eqnarray}
& &\int\mu_{\theta^i\omega}\left(C_{k_n}\right)\mu_{\theta^j\omega}\left(C_{k_n}\right)\mu_{\theta^{i'}\omega}\left({C}_{k_n}'\right)\mu_{\theta^{j'}\omega}\left({C}_{k_n}'\right)d \P\nonumber\\
&\leq & \int(\mu_{\theta^i\omega}\left(C_{k_n}\cap\sigma^{-(j-i)}C_{k_n}\right)+\alpha(g))(\mu_{\theta^{i'}\omega}\left({C}_{k_n}'\cap\sigma^{-(j'-i')}{C}_{k_n}'\right)+\alpha(g))d\P\nonumber\\
&\leq& \alpha(g)^2+\alpha(g)\mu\left(C_{k_n}\cap\sigma^{-(j-i)}C_{k_n}\right)+\alpha(g)\mu\left({C}_{k_n}'\cap\sigma^{-(j'-i')}{C}_{k_n}'\right)\nonumber\\
& &+\int\mu_{\theta^i\omega}\left(C_{k_n}\cap\sigma^{-(j-i)}C_{k_n}\right)\mu_{\theta^{i'}\omega}\left({C}_{k_n}'\cap\sigma^{-(j'-i')}{C}_{k_n}'\right)d\P\nonumber\\
&\leq& \alpha(g)^2+2\alpha(g)+\int(\alpha(g)+\mu_{\theta^i\omega}\left(C_{k_n}\cap\sigma^{-(j-i)}C_{k_n}\cap\sigma^{-(i'-i)}({C}_{k_n}'\cap\sigma^{-(j'-i')}{C}_{k_n}')\right)d\P\nonumber\\
&\leq& \alpha(g)^2+3\alpha(g)+\mu\left(C_{k_n}\cap\sigma^{-(j-i)}C_{k_n}\cap\sigma^{-(i'-i)}({C}_{k_n}'\cap\sigma^{-(j'-i')}{C}_{k_n}')\right)\nonumber\\
&\leq& \alpha(g)^2+4\alpha(g)+\mu\left(C_{k_n}\cap\sigma^{-(j-i)}C_{k_n}\right)\mu\left({C}_{k_n}'\cap\sigma^{-(j'-i')}{C}_{k_n}'\right)\nonumber\\
&\leq& 2\alpha(g)^2+6\alpha(g)+\mu\left(C_{k_n}\right)^2\mu\left({C}_{k_n}'\right)^2.\label{case1}
\end{eqnarray}

Case 2: only two indices are close. We will assume that $i\leq j\leq i'\leq j'$ and that $j-i>g+k_n$, $j-i'>g+k_n$, $j'-i'\leq g+k_n$. Since the cylinders form a partition and that the sample measures are probability measures, we have
\begin{eqnarray}
& &\sum_{C_{k_n},{C}_{k_n}'}\int\mu_{\theta^i\omega}\left(C_{k_n}\right)\mu_{\theta^j\omega}\left(C_{k_n}\right)\mu_{\theta^{i'}\omega}\left({C}_{k_n}'\right)\mu_{\theta^{j'}\omega}\left({C}_{k_n}'\right)d \P\nonumber\\
&\leq& \sum_{C_{k_n}}\int\mu_{\theta^i\omega}\left(C_{k_n}\right)\mu_{\theta^j\omega}\left(C_{k_n}\right)\max_{{C}_{k_n}'}\mu_{\theta^{i'}\omega}\left({C}_{k_n}'\right)d \P.\label{eq2close}
\end{eqnarray}
When the indices are in a different position and/or the two close indices are not $j'$ and $i'$, the same idea can be used. However, one need to choose carefully with which index to take the maximum so that one index disappears with one sum and we obtain a similar term as \eqref{eq2close} where the 3 remaining indices are far from each other. Then, we use the mixing assumptions (III') and (IV') (a similar estimate is obtained when (III) and (IV) are satisfied) to get
\begin{eqnarray}
& & \sum_{C_{k_n}}\int\mu_{\theta^i\omega}\left(C_{k_n}\right)\mu_{\theta^j\omega}\left(C_{k_n}\right)\max_{{C}_{k_n}'}\mu_{\theta^{i'}\omega}\left({C}_{k_n}'\right)d \P\nonumber\\
&\leq& \sum_{C_{k_n}}\left(\int\mu_{\omega}\left(C_{k_n}\right)d\P \int\mu_{\omega}\left(C_{k_n}\right)d\P \int\max_{{C}_{k_n}'}\mu_{\omega}\left({C}_{k_n}'\right)d \P+\rho(g)\xi^{3k_n} \right)\nonumber\\
&=&\rho(g)\xi^{3k_n}N^{k_n}+ \int\max_{{C}_{k_n}'}\mu_{\omega}\left({C}_{k_n}'\right)d \P\sum_{C_{k_n}}\mu\left(C_{k_n}\right)^2.\label{case2}
\end{eqnarray}

Case 3: three indices are close and one is far from them. We will assume that $i\leq j\leq i'\leq j'$ and that $j-i\leq g+k_n$, $j-i'\leq g+k_n$, $j'-i'>g+k_n$. Since $\mu_{\theta^j\omega}\left(C_{k_n}\right)\leq1$ and $\mu_{\theta^i\omega}$ is a probability measure we have
\begin{eqnarray*}
& &\sum_{C_{k_n},{C}_{k_n}'}\int\mu_{\theta^i\omega}\left(C_{k_n}\right)\mu_{\theta^j\omega}\left(C_{k_n}\right)\mu_{\theta^{i'}\omega}\left({C}_{k_n}'\right)\mu_{\theta^{j'}\omega}\left({C}_{k_n}'\right)d \P\\
&\leq& \sum_{{C}_{k_n}'}\int\mu_{\theta^{i'}\omega}\left({C}_{k_n}'\right)\mu_{\theta^{j'}\omega}\left({C}_{k_n}'\right)d \P.
\end{eqnarray*}
When the indices are in a different position, one can use the same idea so that we stay with two indices which are far from each other and measure the same cylinder. Thus we can use the mixing assumptions (I) an (II), to obtain
\begin{eqnarray}
\sum_{{C}_{k_n}'}\int\mu_{\theta^{i'}\omega}\left({C}_{k_n}'\right)\mu_{\theta^{j'}\omega}\left({C}_{k_n}'\right)d \P
&\leq& \sum_{{C}_{k_n}'}\int\left(\mu_{\theta^{i'}\omega}\left({C}_{k_n}'\cap\sigma^{(j'-i')}{C}_{k_n}'\right)+\alpha(g)\right)d \P\nonumber\\
 &\leq&\sum_{{C}_{k_n}'}\left(\mu\left({C}_{k_n}'\cap\sigma^{(j'-i')}{C}_{k_n}'\right)+\alpha(g)\right)\nonumber\\
 &\leq&\alpha(g)N^{k_n}+\sum_{{C}_{k_n}'}\mu\left({C}_{k_n}'\right)^2.\label{case3}
 \end{eqnarray}
 
Case 4: two indices are close and both are far from the two other indices which are close from one another. We will assume that $i\leq j\leq i'\leq j'$ and that $j-i\leq g+k_n$, $j-i'> g+k_n$, $j'-i'\leq g+k_n$.  Since the sample measures are probability measures, we obtain
\begin{eqnarray}
& &\sum_{C_{k_n},{C}_{k_n}'}\int\mu_{\theta^i\omega}\left(C_{k_n}\right)\mu_{\theta^j\omega}\left(C_{k_n}\right)\mu_{\theta^{i'}\omega}\left({C}_{k_n}'\right)\mu_{\theta^{j'}\omega}\left({C}_{k_n}'\right)d \P\nonumber\\
&\leq&\int\max_{{C}_{k_n}}\mu_{\theta^{j}\omega}\left({C}_{k_n}\right)\max_{{C}_{k_n}'}\mu_{\theta^{i'}\omega}\left({C}_{k_n}'\right)d \P.\label{eq22}
\end{eqnarray}
For the other relative positions, we can observe that
\begin{itemize}
\item if the measures with the two indices that are far from each other measure different cylinders, we obtain an estimate similar to \eqref{eq22};
\item if the measures with the two indices that are far from each other measure the same cylinder, the case can be treat as case 3.
\end{itemize}
Then, using the mixing assumptions (III') and (IV') (a similar estimate is obtained when (III) and (IV) are satisfied), we have
\begin{eqnarray}
\int\max_{{C}_{k_n}}\mu_{\theta^{j}\omega}\left({C}_{k_n}\right)\max_{{C}_{k_n}'}\mu_{\theta^{i'}\omega}\left({C}_{k_n}'\right)d \P
&\leq&\rho(g)\xi^{2k_n}+ \left(\int\max_{{C}_{k_n}}\mu_{\omega}\left({C}_{k_n}\right)d \P\right)^2.\label{case4}
\end{eqnarray}

Case 5: all the indices are close. We will assume that $i\leq j\leq i'\leq j'$ and that $j-i\leq g+k_n$, $j-i'\leq g+k_n$, $j'-i'\leq g+k_n$. In this case, the relative position is irrelevant. Since the sample measures are probability measures, we obtain
\begin{equation}\label{case5}\sum_{C_{k_n},{C}_{k_n}'}\int\mu_{\theta^i\omega}\left(C_{k_n}\right)\mu_{\theta^j\omega}\left(C_{k_n}\right)\mu_{\theta^{i'}\omega}\left({C}_{k_n}'\right)\mu_{\theta^{j'}\omega}\left({C}_{k_n}'\right)d \P
\leq\int\max_{{C}_{k_n}}\mu_{\omega}\left({C}_{k_n}\right)d \P.
\end{equation}

Finally, \eqref{eqcheb} together with \eqref{lowerint}, \eqref{case1}, \eqref{case2}, \eqref{case3}, \eqref{case4} and \eqref{case5} gives us that there exists a constant $c_1$ such that

\begin{eqnarray}
& &\P\left(\bigg|\EE_\omega(S_n)-n^2\sum_{C_{k_n}}\mu(C_{k_n})^2\bigg|\geq(n^2\sum_{C_{k_n}}\mu(C_{k_n})^2)^\delta\right)\nonumber\\
&\leq&\frac{c_1}{\left(n^2\sum_{C_{k_n}}\mu\left(C_{k_n}\right)^{2}\right)^{2\delta}}\Bigg(n^3(g+k_n)\int\max_{{C}_{k_n}'}\mu_{\omega}\left({C}_{k_n}'\right)d \P\sum_{C_{k_n}}\mu\left(C_{k_n}\right)^2+n^2(g+k_n)^2\sum_{{C}_{k_n}'}\mu\left({C}_{k_n}'\right)^2\nonumber\\
& &+ n^2(g+k_n)^2 \bigg(\int\max_{{C}_{k_n}}\mu_{\omega}\left({C}_{k_n}\right)d \P\bigg)^2+n(g+k_n)^3\int\max_{{C}_{k_n}}\mu_{\omega}\left({C}_{k_n}\right)d \P \nonumber\\
 & &+n^3(g+k_n)\bigg(\sum_{C_{k_n}}\mu\left(C_{k_n}\right)^2\bigg)^2+ \Gamma(g)\Bigg)\label{ineqbigsum}
\end{eqnarray}
where
\begin{eqnarray*}
\Gamma(g)&=&n^4(\alpha(g)^2+\alpha(g))+n^3(g+k_n)\rho(g)\xi^{3k_n}N^{k_n}+n^2(g+k_n)^2\alpha(g)N^{k_n}\\ 
& &+ n^2(g+k_n)^2\rho(g)\xi^{2k_n}+n^4\alpha(g)N^{k_n}\sum_{C_{k_n}}\mu\left(C_{k_n}\right)^2.
\end{eqnarray*}
We recall that $k_n=\frac{1}{\overline{H}_2(\mu)+\eps}(2\log n+d\log\log n)$ and that $g=\log(n^\beta)$. Thus, as in \eqref{estmax}, we have
\[n\int_\Omega\underset{C_{k_n}}\max\,\mu_{\omega}\left(C_{k_n}\right)d\P(\omega)\leq n^{-\eps}\]
for every $n$ large enough.

Moreover, by definition of $\overline{H}_2(\mu)$ and our choice of $k_n$, we have for every $n$ large enough
\[n^2\sum_{C_{k_n}}\mu\left(C_{k_n}\right)^{2}\geq \left(\log n\right)^{-d}.\]
First of all, we choose $\beta\gg1$ so that for any n large enough
\[\Gamma(g)\leq 1.\]
Then, for the first term in \eqref{ineqbigsum}, we have
\begin{eqnarray*}
\frac{n^3(g+k_n)\int\max_{{C}_{k_n}'}\mu_{\omega}\left({C}_{k_n}'\right)d \P\sum_{C_{k_n}}\mu\left(C_{k_n}\right)^2}{\left(n^2\sum_{C_{k_n}}\mu\left(C_{k_n}\right)^{2}\right)^{2\delta}}
&=&\frac{n(g+k_n)\int\max_{{C}_{k_n}'}\mu_{\omega}\left({C}_{k_n}'\right)d \P}{\left(n^2\sum_{C_{k_n}}\mu\left(C_{k_n}\right)^{2}\right)^{2\delta-1}}\\
&\leq&\frac{(g+k_n)n^{-\eps}}{\left(\log n\right)^{-b(2\delta-1)}}.
\end{eqnarray*}
For the second term in \eqref{ineqbigsum}, we have
\begin{eqnarray*}
\frac{n^2(g+k_n)^2\sum_{{C}_{k_n}'}\mu\left({C}_{k_n}'\right)^2}{\left(n^2\sum_{C_{k_n}}\mu\left(C_{k_n}\right)^{2}\right)^{2\delta}}
&=&\frac{(g+k_n)^2}{\left(n^2\sum_{C_{k_n}}\mu\left(C_{k_n}\right)^{2}\right)^{2\delta-1}}\\
&\leq&\frac{(g+k_n)}{\left(\log n\right)^{-b(2\delta-1)}}.
\end{eqnarray*}
For the third term in \eqref{ineqbigsum}, we have
\begin{eqnarray*}
\frac{n^2(g+k_n)^2 \bigg(\int\max_{{C}_{k_n}}\mu_{\omega}\left({C}_{k_n}\right)d \P\bigg)^2}{\left(n^2\sum_{C_{k_n}}\mu\left(C_{k_n}\right)^{2}\right)^{2\delta}}
&\leq&\frac{n^{-2\eps}(g+k_n)^2}{\left(\log n\right)^{-2b\delta}}.
\end{eqnarray*}
For the fourth term in \eqref{ineqbigsum}, we have
\begin{eqnarray*}
\frac{n(g+k_n)^3\int\max_{{C}_{k_n}}\mu_{\omega}\left({C}_{k_n}\right)d \P}{\left(n^2\sum_{C_{k_n}}\mu\left(C_{k_n}\right)^{2}\right)^{2\delta}}
&\leq&\frac{n^{-\eps}(g+k_n)^3}{\left(\log n\right)^{-2b\delta}}.
\end{eqnarray*}
And, for the fifth term in \eqref{ineqbigsum}, we have
\begin{eqnarray*}
\frac{n^3(g+k_n)\bigg(\sum_{C_{k_n}}\mu\left(C_{k_n}\right)^2\bigg)^2}{\left(n^2\sum_{C_{k_n}}\mu\left(C_{k_n}\right)^{2}\right)^{2\delta}}
&=&\frac{(g+k_n)\bigg(\sum_{C_{k_n}}\mu\left(C_{k_n}\right)^{2}\bigg)^{1/2}}{\left(n^2\sum_{C_{k_n}}\mu\left(C_{k_n}\right)^{2}\right)^{2\delta-3/2}}\\
&\leq&\frac{(g+k_n)}{\left(\log n\right)^{-b(2\delta-3/2)}}.
\end{eqnarray*}
Finally, putting all these estimates together in \eqref{ineqbigsum}, choosing $d\ll -1$ and since $3/4<\delta<1$, we obtain
\[\P\left(\bigg|\EE_\omega(S_n)-n^2\sum_{C_{k_n}}\mu(C_{k_n})^2\bigg|\geq(n^2\sum_{C_{k_n}}\mu(C_{k_n})^2)^\delta\right)=\mathcal{O}(\frac{1}{\log n}).\]
\end{proof}


\textbf{{Acknowledgements:}} The author would like to thank Rodrigo Lambert for various comments on a first draft of the paper, Mike Todd for fruitful discussions and for fixing the mistake found in \cite{RT} and the referee for useful suggestions to improve the paper.

\end{document}